\documentclass{amsart}
\usepackage{amsfonts,latexsym,amssymb,amsmath,amsthm,color}
\usepackage{enumerate}
\usepackage{pb-diagram}

\newtheorem{teorema}{Theorem}[section]
\newtheorem{propo}[teorema]{Proposition}
\newtheorem{lema}[teorema]{Lemma}
\newtheorem{coro}[teorema]{Corollary}

\theoremstyle{definition}
\newtheorem{defin}[teorema]{Definition}

\theoremstyle{remark}
\newtheorem{nota}[teorema]{Remark}
\newtheorem{ejemplo}[teorema]{Example}

\newcommand{\be}{\begin{enumerate}}
\newcommand{\ee}{\end{enumerate}}

\newcommand{\bi}{\begin{itemize}}
\newcommand{\ei}{\end{itemize}}

\newcommand{\RR}{{\mathbb R}}

\newcommand{\NN}{{\mathbb N}}
\newcommand{\QQ}{{\mathbb Q}}
\newcommand{\CC}{{\mathbb C}}

\newcommand{\FF}{{\mathcal F}}

\newcommand{\OO}{{\mathcal O}}
\newcommand{\XX}{{\mathcal X}}

\newcommand{\om}{\omega}

\newcommand{\bx}{\mathbf{x}}

\newcommand{\bv}{{\mathbf v}}

\newcommand{\bo}{{\mathbf 0}}

\DeclareMathOperator{\Sing}{Sing}

\DeclareMathOperator{\codim}{codim}

\begin{document}

\title{On codimension one foliations with prescribed cuspidal separatrix}

\author{Percy Fern\'{a}ndez-S\'{a}nchez}
\address[Percy Fern\'{a}ndez S\'{a}nchez]{Dpto. Ciencias - Secci\'{o}n Matem\'{a}ticas \\ Pontificia Universidad Cat\'{o}lica del Per\'{u} \\ Av. Universitaria 1801 \\ San Miguel, Lima 32 \\ Per\'{u}.}
\email{pefernan@pucp.edu.pe}
\thanks{First author partially supported by the Pontificia Universidad Catolica del Per\'{u} project VRI-DGI 2013-0047.}

\author{Jorge Mozo Fern\'{a}ndez}
\address[Jorge Mozo Fern\'{a}ndez]{Dpto. Matem\'{a}tica Aplicada \\
Facultad de Ciencias \\ Campus Miguel Delibes \\
Universidad de Valladolid\\
Paseo de Bel\'{e}n, 7\\
47011 Valladolid\\
Spain.}
\email{jmozo@maf.uva.es}
\thanks{Second and third authors partially supported by the Spanish national project MTM2010-15471}

\author{Hern\'{a}n Neciosup}
\address[Hern\'{a}n Neciosup]{Dpto. An\'{a}lisis Matem\'{a}tico, \'{A}lgebra, Geometr\'{\i}a y Topolog\'{\i}a \\
Facultad de Ciencias \\ Campus Miguel Delibes \\
Universidad de Valladolid\\
Paseo de Bel\'{e}n, 7\\
47011 Valladolid\\
Spain.}
\email{hneciosup@agt.uva.es}
\thanks{Third author supported by a PhD grant of the program \emph{FPI-UVa - Ayudas UVa para la Formaci\'{o}n del Personal Investigador} (University of Valladolid)}



\date{\today}


\begin{abstract}
In this paper, we will construct a pre-normal form for germs of codimension one holomorphic foliation having a particular type of separatrix, of cuspidal type. We will also give a sufficient condition, in the quasi-homogeneous, three-dimensional case, for these foliations to be of generalized surface type.
\end{abstract}

\maketitle

\section{Preliminaries and statement of the results}

The objective of this paper is twofold: first, we will find a (pre)-normal form for holomorphic differential 1-forms defining local holomorphic foliations in $(\CC^{n+1},\bo )$,  with a fixed separatrix of a certain particular
 type, that we will describe later, and with a condition on the multiplicity at the origin. Second, we will give a sufficient condition in dimension three for some of these forms, that will be called cuspidal, in order they are of the generalized surface type (case in which the previous condition on the order is automatically verified).

The search of normal forms for germs of holomorphic vector
fields, or of holomorphic 1-forms, is carried out in several
papers, either treating directly this problem, or as a need  for
treating another problems, such as the formal, analytic or
topological classification of these objects. Let us define the main concepts that we will use throughout the paper, and a brief account of the main achievements in dimensions two and three, focusing in the case of codimension one foliations.

A germ of codimension one holomorphic foliation $\FF$ in $(\CC^n,\bo)$ is defined by a holomorphic 1-form $\omega$, satisfying the Frobenius integrability condition $\omega \wedge d\omega=0$, and such that its coefficients have no common factors. If $\omega = \sum\limits_
{i=1}^n a_i (\bx ) dx_i$, the \emph{singular set} of $\omega$, $\Sing (\omega )$, is the germ of analytic set defined by the zeros of the ideal $(a_i (\bx )_{1\leq i \leq n})$. Previous condition implies that $\Sing (\omega )$ has codimension at least two.

We will also consider meromorphic integrable 1-forms: these are 1-forms $\omega = \sum\limits_{i=1}^n a_i (\bx ) dx_i$, where $a_i (\bx )$ are germs of meromorphic functions. Consider a $(n-1)$-dimensional germ of analytic set, $S$, defined by an equation
 $(f=0)$, with $f\in \OO$ reduced (here and throughout the paper, $\OO=\OO_n$ will denote the ring $\CC \{ \bx \}$ of convergent power series in $n$ variables, where $n$ is usually omitted as there is no risk of confusion).

A meromorphic 1-form is called \emph{logarithmic} along $S$ if $f\omega$ and $fd\omega$ are holomorphic forms, or equivalently if $f\omega$ and $df \wedge \omega$ are holomorphic. The following result follows:
\begin{propo}[\cite{LN,Saito80}]
Let $\omega$ be a holomorphic 1-form, and $f$ as before. The following conditions are equivalent:
\be
\item There exist a holomorphic 2-form $\eta$ such that $\omega \wedge df= f\eta$.
\item $\omega /f$ is logarithmic along $S$.
\item There exist $g$, $h\in\OO$, and a holomorphic 1-form $\alpha$, such that $g\omega+hdf=f\alpha$, and moreover, $g$, $f$ have no common factors.
\ee
\end{propo}

The equivalence between (1) and (3) can be read in \cite{LN}, assuming irreducibility of $f$. The result is stated in \cite{Saito80} in a more general context ($q$-forms), without the irreducibility assumption, but the proof provided there is only valid in the irreducible case. Even if it is a well-known result, for the sake of completeness, we will write a complete proof here.

\begin{proof}
$df\wedge \dfrac{\omega}{f} = \eta$ if and only if $df\wedge \omega =f\eta$. This shows that (1) and (2) are equivalent.

Assuming (3), $g\omega\wedge df= f\cdot \alpha\wedge df$. As $f$, $g$ have no common factors, $f$ divides $\omega \wedge df$, so (1) follows. Conversely, assume first that, for some $k$, $f$ and $\dfrac{\partial f}{\partial x_k}$ have no common factors. If  $\omega = \sum\limits_{i=1}^n a_i (\bx ) dx_i$, (1) implies that there exist $g_{ij}(\bx )\in \OO$ such that
$$
a_i (\bx )\frac{\partial f}{\partial x_j} - a_j (\bx ) \frac{\partial f}{\partial x_i}= f g_{ij} (\bx),
$$
for every $1\leq i\leq j\leq n$. So,
$$
\frac{\partial f}{\partial x_k} \cdot \omega = \sum_{i=1}^n \frac{\partial f}{\partial x_k} \cdot a_i (\bx ) dx_i = a_k (\bx )\, df +f\sum_{i=1}^n g_{ik} (\bx ) dx_i,
$$
and (3) follows. This shows the result in the irreducible case. In the general case, there is a directional derivative $D_\bv f= \sum\limits_{i=1}^n v_i \dfrac{\partial f}{\partial x_i}$ ($\bv=(v_1,v_2,\ldots , v_n)$) without common factors with $f$. If it were not the case, there will exist $\bv_1,\bv_2,\ldots , \bv_n$ linearly independent, and a factor $f'$ of $f$, such that $f'$ divides each $D_{\bv_i} f$. So, $f'$ is a factor of every partial derivative of $f$, which is not possible as $f$ is reduced.
 So, if $D_\bv f$ has no common factors with $f$, a linear change of variables allows to assume that $\dfrac{\partial f}{\partial x_1}$ has no common factors with $f$, and the previous case applies.
\end{proof}

If $\FF$ is a holomorphic foliation defined by a 1-form $\omega$ satisfying the conditions of the previous Theorem, we will say that $S$ (or $f=0$) is a separatrix for $\FF$. Let us observe that, if $\Omega^1 (\log S)$ represents the set of meromorphic 1-
forms having $f$ as separatrix (i.e. logarithmic along $S$), $\Omega^1 (\log S)$ has a structure of $\OO$-module. Under some condition, this $\OO$-module is free. In fact, K. Saito proves in \cite{Saito80} the following result:

\begin{propo} \label{modulolibre}
$\Omega^1 (\log S)$ is a free $\OO$-module if and only if there exist $\omega_1,\omega_2,\ldots , \omega_n \in \Omega^1 (\log S)$ such that $\omega_1\wedge \cdots \wedge \omega_n= U(\bx ) \dfrac{dx_1\wedge \cdots \wedge dx_n}{f}$, where $U(\bx )$ is a unit in $\OO$.
\end{propo}

When this condition is fulfilled, and $\omega_1,\omega_2, \ldots , \omega_n$ are found, it is easy to represent the elements of $\Omega^1 (\log S)$.

\begin{ejemplo}
Let $f=y^2+x^n$, $\omega_1=\frac{df}{f}$, $\omega_2 =\frac{1}{f} (2x dy-nydx)$, $\omega_1\wedge \omega_2=\frac{2n}{f} dx\wedge dy$. So, $\{ \omega_1,\omega_2\}$ is a basis of the free $\OO$-module $\Omega^1 (\log S)$.
\end{ejemplo}

In this paper, we are interested in finding explicitly generators of a foliation having particular separatrices, and satisfying certain additional conditions that we will precise. The preceding results are not of immediate application, because, for instance, they don't take into account the integrability condition if $n\geq 3$; so, we will need to use different techniques. We are interested in the problem of classifying analytically germs of holomorphic foliations of cuspidal quasi-homogeneous type in dimension three, and the foliations considered in the first part of this paper are a $n+1$-dimensional generalization of these. We will recall first some of the main notions and known results in dimensions two and three.

Let $\FF_1$, $\FF_2$ be two germs of holomorphic foliations in
$(\CC^n,\bo )$, represented by integrable 1-forms $\om_1$, $\om_2$
respectively. We will say that $\FF_1$, $\FF_2$ are analytically
equivalent if there exist a biholomorphism $\Phi:
(\CC^n,\bo )\rightarrow (\CC^n,\bo )$ such that  $\Phi^{\ast} \om_1\wedge\om_2=0$, or equivalently, if there exist a unit $U$ such that $\Phi^{\ast} \om_1 =
U\cdot \om_2$. If we only impose $\Phi$ to
be formal (i.e., $\Phi:\CC [[\bx ]]^n \rightarrow \CC[[\bx ]]^n$ invertible), we will say that $\FF_1$, $\FF_2$ are formally
equivalent.

A germ of two-dimensional holomorphic foliation $\FF$
has a simple singularity at the origin if it can be defined by a
1-form $\om = a(x,y)dx+b(x,y)dy$, such that the linear part of
$b(x,y)\frac{\partial}{\partial x} -
a(x,y)\frac{\partial}{\partial y}$ (the dual vector field) has
two eigenvalues $\lambda_1$, $\lambda_2$, with $\lambda_2\neq
0$, $\lambda_1/\lambda_2\notin \QQ_{<0}$. If $\lambda_1/\lambda_2
\notin \QQ_{\leq 0}\cup \NN \cup \frac{1}{\NN}$, the foliation is
formally linearizable, i.e., formally equivalent to $\lambda_2y
dx-\lambda_1 xdy$. Moreover, if $\lambda_1/\lambda_2\notin
\RR_{\leq 0} \cup \NN \cup \frac{1}{\NN}$, it is analytically
linearizable.

The analytic classification in the cases $\lambda_1/\lambda_2 \in
\QQ_{\leq 0}$ has been studied by J.~Martinet and J.-P.~Ramis
\cite{MR1,MR2}. When $\lambda_1/\lambda_2=-q/p\in \QQ_{<0}$,
$p\wedge q=1$, the foliation is formally equivalent to $qx(1+\mu\cdot
(x^py^q)^s) dy +py(1+(\mu+1)\cdot (x^py^q)^s)dx$, and analytically
equivalent to $pydx +qx (1+ a(x,y))dy$. If $\lambda_1=0$, it is
formally equivalent to $x^{p+1} dy -y (1+\lambda x^p) dx$, and
analytically equivalent to $x^{p+1}dy -A(x,y)dx$, with
$A(0,y)=y$. Let us observe that, in this last case, the formal expression $x^{p+1} dy -y (1+\lambda x^p) dx$ is unique, so $(p,\lambda )$ is a complete system of formal invariants. We will speak for formal normal form due to this uniqueness. The expression we have written in the analytic setting is not unique, so the word ``normal'' is not appropriate here: we will speak instead of \emph{prenormal forms}.

Consider now a nilpotent foliation, i.e., defined by a 1-form
such that the dual vector field has a nilpotent, non zero, linear
part, i.e., $\om= ydy + \cdots $. \mbox{F.~Takens}
\cite{Takens} shows that such a foliation is formally equivalent
to $\om_{n,p}=d(y^2+x^n)+x^p U(x)dy$, for some integers $n\geq
3$, $p\geq 2$, and $U(x)\in \CC[[x]]$, $U(0)\neq 0$. These integers are not uniquely determined: nevertheless the conditions $2p>n$, $2p=n$, $2p<n$ are preserved under conjugation. If $2p>n$, the foliation has
an invariant curve analytically equivalent to the cusp $y^2+x^n=0$
(for this reason these foliations are usually called
\textit{cuspidal}). In the study of
these foliations, it is interesting to write down explicitly the
foliation having $y^2+x^n=0$ as a separatrix. Following Proposition \ref{modulolibre}, such a foliation can be defined by a form $a(x,y) d(y^2+x^n)+b(x,y) (nydx-2xdy)$. The condition about the order at the origin implies that $a(x,y)$ is a unit, so $\omega = d(y^2+x^n)+A(x,
y) (nydx-2xdy )$ defines such a foliation, as is stated in \cite{Cerveau-Moussu}. This form is very useful in order to study the
analytic classification, via the projective holonomy, of cuspidal
foliations (see
\cite{Cerveau-Moussu,Loray-Meziani,Meziani,BMS,Strozyna,Meziani-Sad}).

Let us consider now the three-dimensional case. Reduction of singularities is shown in \cite{Cano-Cerveau,Cano}. This means that a birational map $\pi: \tilde{M}\rightarrow (\CC^3,\bo )$ can be constructed, composition of permissible blow-ups, such that the foliation $\pi^\ast \FF$ has only a kind of singularities called simple (see \cite{Cano-Cerveau,Cano} for precise definitions and descriptions of simple singularities in dimension three). These blow-ups construct a divisor in $\tilde{M}$, called the \emph{exceptional divisor}, with a finite number of irreducible components. For one of this simple singularities, call it $P$, the dimensional type is defined as
$$
\codim_{\CC} \{ \XX (P);\ \XX \text{ is a vector field in } (\CC^3,\bo ) \text{ tangent to }\FF \} .
$$
Equivalently, it is the minimum number of coordinates needed to describe a 1-form that defines the foliation. Singularities of dimensional type equal to two are cylinders over two-dimensional singularities, and their study is reduced to the planar case. Singularities of dimensional type three have formal normal forms as follows:

\be
\item $\om= xyz \left( \lambda_1\frac{dx}{x}+\lambda_2
\frac{dy}{y}+ \lambda_3 \frac{dz}{z}\right)$, in the linearizable
case (the quotients between the $\lambda_i$ do not belong to
$\QQ_{<0}$).
\item $\om = xyz(x^py^qz^r)^s \left[ \alpha \frac{dx}{x} + \beta
\frac{dy}{y}+ \left( \lambda +\frac{1}{(x^py^qz^r)^s}\right)
\cdot \left( p\frac{dx}{x}+ q\frac{dy}{y}+ r\frac{dz}{z}\right)
\right]$, $p$, $q$, $z\in \NN$, $s\in \NN^{\ast}$, $\alpha$,
$\beta \in \CC$ (resonant case).
\ee

The study of the analytic classification of simple singularities of  foliations  is done in \cite{Cerveau-LN} and \cite{Cerveau-Mozo}.

Consider non-simple singularities. We are interested in a convenient generalization  of cuspidal foliations, namely, foliations having $z^2+\varphi (x,y)=0$ as its separatrix. These surfaces are a particular case of the sometimes called \emph{Zariski surfaces}, of equations $z^k+\varphi (x,y) =0$, which, from different points of view, have been largely studied by several authors (see, for instance, \cite{Pichon} or \cite{MN}). The analytic classification in the \textit{quasi-ordinary} case, i.e., when $\varphi (x,y)
=x^py^q$ (the discriminant of the projection over the plane $(x,y)$ has normal crossings), is studied in detail in \cite{FM}, under some additional assumptions. In that paper,  a precise expression of foliations having $z^2+x^py^q=0$ as a separatrix is needed (Proposition 1 from \cite{FM}), and it is computed directly.

In this paper, we will describe a pre-normal form for singularities of codimension one holomorphic foliations in higher dimension, of cuspidal type. By \emph{cuspidal} we will understand non-dicritical, generalized hypersurfaces, having $z^2+\varphi
(\bx)=0$ as a separatrix. Before stating the main result, we will recall briefly some definitions, again in dimension three, and we will generalize these notions to higher dimensions.

A three-dimensional germ of holomorphic foliation is called non-dicritical if in its reduction of singularities, all components of the exceptional divisor that appear are generically leaves of the foliation. It is called \emph{generalized surface} if it is non-dicritical and moreover, no saddle-nodes appear after their reduction. Equivalently, if every non-degenerate plane section is a generalized curve in the sense of \cite{CLNS}. In \cite{Cano-Cerveau} it is shown that every non-dicritical three-dimensional singularity of holomorphic codimension one foliation admits a separatrix (in fact, a finite set of them). Generalized surfaces are studied in \cite{FM2}. We recall here the properties of generalized surfaces we will need throughout the paper. For, call $f$
a reduced equation of its set of separatrices. Then:
\be
\item The reduction of the singularities of $\FF$ agrees with the strong, embedded resolution of the singularities of $f$.
\item Write $\om= \om_\nu + \om_{\nu +1} + \cdots$, where the coefficients of $\om_k$ are homogeneous polynomials of degree $k$, and $\om_\nu\not\equiv 0$. The number $\nu$ is the \emph{order }   of $\om$ at the origin, and it is denoted by $\nu_0(\om )$.
 Then $\nu_0 (\om)=\nu_0 (df)$.
\ee

This notion can be extended to higher dimensions. As no reduction of singularities is known in the general case when $n\geq 3$, the definition is done by transversal sections.

\begin{defin}
A germ of foliation in $(\CC^n,\bo )$ is a \emph{generalized hypersurface} if every non-degenerate transversal plane section is a non-dicritical generalized curve.

Equivalently, by recurrence on $n$, if every non-degenerate transversal hyperplane section is a generalized hypersurface in $(\CC^{n-1},\bo )$.
\end{defin}

Recall, from \cite{Mattei-Moussu} and \cite{FM2} that if $i: (\CC^2,\bo )\rightarrow (\CC^n,\bo )$ is a transversal plane section, the following conditions holds:
\be
\item ${\rm Sing} (i^\ast \om )=i^{-1} ({\rm Sing} (\om ))=\{ \bo\}$.
\item $\nu_0 (\FF )= \nu_0 (i^\ast \FF )$.
\ee

In particular, this means that if $f=0$ is an equation of the set of separatrices of $\FF$, then $\nu_0 (df )= \nu_0 (\om )$. In fact, results from \cite{Cano-Cerveau,Cano-Mattei} show that every separatrix of $i^\ast \FF$ extends to a separatrix of $\FF
$, and moreover we have that $\nu_0 (\om)=\nu_0 (i^\ast \FF ) =\nu_0 (d(f\circ i))=\nu_0 (df)$.

As for generalized surfaces in the three-dimensional case, the following result can be proved for generalized hypersurfaces.
\begin{teorema}
The reduction of singularities of a generalized hypersurface is the same as the reduction of its set of separatrices.
\end{teorema}

As in smaller dimensions, the key point is the following Lemma:
\begin{lema}
If a generalized hypersurface has exactly $n$ smooth and transversal separatrices through a singularity $P$, then it is simple.
\end{lema}
\begin{proof}
In fact, take a transversal section $i$ through $P$, such that $i^\ast \FF$ has $n$ smooth transversal curves as separatrices, so $\nu_0 (i^\ast \om )= n-1= \nu_0 (\om )$. If the foliation is generated by a 1-form ${\displaystyle \om= x_1x_2\cdots x_n \sum_{i=1}^n a_i (\bx )\frac {dx_i}{x_i}}$, this implies that some $a_i (\bo )\neq 0$, so the singularity is pre-simple. A pre-simple generalized hypersurface, that is a corner, is simple.
\end{proof}

In this paper, we will consider germs of holomorphic foliations in $(\CC^{n+1}, \bo )$, where coordinates are denoted $(\bx, z)= (x_1,x_2,\ldots , x_n, z)$. A cuspidal hypersurface will be defined by an analytic equation $f=z^2+\varphi (\bx )=0$, where $\nu_0 (\varphi )\geq 2$. We will consider foliations $\FF$ having such $f$ as a separatrix, and of the generalized hypersurface type, as described above. Such a foliation will also be called cuspidal.

    The first result of the paper is:

 \begin{teorema} \label{separatriz}
Consider a cuspidal foliation $\FF$ on $(\CC^{n+1}, \bo )$, with separatrix defined by the equation $f=z^2+\varphi (\bx )=0$. Then, there exist
coordinates such that  a generator of $\FF$ is
$$
\om = d(z^2+\varphi') + G(\Psi, z)\cdot (z\cdot \Psi) \left(
2\frac{dz}{z}-\frac{d\varphi'}{\varphi'}\right) ,
$$
where $\varphi' = \varphi\cdot u=\Psi^r$, $u$ a unit, $\Psi$ is not a power, and $G$ is a
germ of holomorphic function in two variables.
\end{teorema}

In the quasi-homogeneous case, $\varphi'$ can be replaced by $\varphi$ in the statement of the Theorem. This result can be seen as a variant of Saito's Theorem, where integrability condition is taken into account. In fact, it would be interesting to characterize all hypersurfaces $f=0$ such that the set of integrable 1-forms having $f$ as a separatrix is a $\OO$-free module.

The proof is based in direct computation, and integrability condition will be introduced geometrically via a Theorem of F. Loray \cite{Loray}, that essentially generalizes Weierstrass Preparation Theorem to the case of germs of holomorphic foliations. As
a final conclusion, we can see that the integrable 1-forms defining a germ of such a cuspidal foliations are pull-backs of two-dimensional forms, modulo a certain unit factor that we need to add to the expression. In the quasi-homogeneous case, this factor can even be avoided, thanks to a property of quasi-homogeneous germs taken from Cerveau and Mattei \cite{Cerveau-Mattei}.

Foliations generated by a 1-form of the type described above can even be dicritical, and have more separatrices. In the last section, we will find a sufficient condition, in dimension three, for one of these foliations, of the quasi-homogeneous cuspidal type, to be a generalized surface, with separatrix exactly given by the equation $z^2+ \prod_{i=1}^l (y^p-a_i x^q)^{d_i}$. This condition will be stated in Theorem \ref{superficie-generalizada}, and is inspired by the one given by Loray \cite{Loray99} in dimension two.

It is worth to mention here that in the list of \emph{Open questions and related problems} proposed by F. Cano and D. Cerveau in \cite{Cano-Cerveau}, one of them is stated as follows:
\begin{quote}
(5) Classify the non-dicritical singular foliations in $(\CC^3, \bo )$ generated by one 1-form with initial part of the type $xdx$.
\end{quote}

As few results about three (or higher)-dimensional holomorphic foliations exist
in literature, we consider that the results in this paper are interesting in order to follow this direction.  Also,  to develop the analytic classification of foliations of cuspidal type, via the projective holonomy, it is necessary to construct a Hopf fibration allowing to lift the conjugation between the holonomies to a conjugation between foliations, and the pre-normal form obtained here is useful for that aim. This will be the subject of a future work of the authors and of the PhD Thesis of the third author.

\section{Proof of Theorem \ref{separatriz}}
We will begin establishing the following Lemma, in a more general setting that the main result of the paper:
\begin{lema} \label{cuentasiniciales}
Let $\omega$ be a germ of 1-form in $(\CC^{n+1},\bo )$, logarithmic along $S$, hypersurface defined by $f= z^k+\varphi (\bx )$, with $\nu_0 (\varphi )\geq k$, and such that $\nu_0 (\omega )= \nu_0 (df)=k-1$. Then, there exist a unit $U$ such that $U\cdot
\omega=\om_1+H\om_2+(z^k+\varphi )\cdot \om_3$, where,
\begin{align*}
\om_1 & = d(z^k+\varphi ),\\
\om_2 & = zd\varphi -k\varphi dz,\\
\om_3 & = \sum_{i=2}^n g_i dx_i.
\end{align*}
\end{lema}

\begin{proof}
Denote $\om= \sum_{i=1}^n A_i dx_i + A dz$. As $\om$, $\om_1$, $\om_2$ are logarithmic along $S$, we have the following relations:
\begin{equation*}
\begin{split}
\om\wedge  \om_1 & = (z^k+\varphi )\cdot \left[ \sum_{i<j} H_{ij} \cdot dx_i\wedge dx_j + \sum_{i=1}^n H_i\cdot dx_i\wedge dz \right] \\
 & = \sum_{i<j} (A_i \varphi_{x_j} - A_j \varphi_{x_i} ) dx_i \wedge dx_j + \sum_{i=1}^n (kz^{k-1} A_i -A\varphi_{x_i}) dx_i\wedge dz. \\
 \om\wedge  \om_2 & = (z^k+\varphi )\cdot \left[ \sum_{i<j} G_{ij} \cdot dx_i\wedge dx_j + \sum_{i=1}^n G_i\cdot dx_i\wedge dz \right] \\
  & = \sum_{i<j} (A_i z\varphi_{x_j}- A_j z\varphi_{x_i}) dx_i\wedge dx_j - \sum_{i=1}^n (k A_i\varphi +zA\varphi_{x_i} ) dx_i \wedge dz.
\end{split}
\end{equation*}

Identifying coefficients, we have systems
$$
\begin{pmatrix}
kz^{k-1} & -\varphi_{x_i} \\
-k\varphi & -z\varphi_{x_i}
\end{pmatrix}
\begin{pmatrix}
A_i \\ A
\end{pmatrix} =
(z^k+\varphi) \begin{pmatrix} H_i \\ G_i \end{pmatrix}.
$$
Pre-multiplying by the adjoint matrix,
$$
-k\varphi_{x_i} (z^k+\varphi)
\begin{pmatrix} A_i \\ A \end{pmatrix} =
(z^k+\varphi)
\begin{pmatrix}
-z\varphi_{x_i} & \varphi_{x_i} \\
k\varphi & kz^{k-1}
\end{pmatrix}
\begin{pmatrix} H_i \\ G_i \end{pmatrix},
$$
so,
\begin{equation*}
\begin{split}
-kA_i & = -zH_i+G_i \\
-\varphi_{x_i} A &= \varphi H_i+z^{k-1}G_i.
\end{split}
\end{equation*}

Using the equality
$$
\varphi_{x_1} \cdot (\varphi H_i+z^{k-1} G_i ) =
\varphi_{x_i} \cdot (\varphi H_1+z^{k-1} G_1 ),
$$
we have that
\begin{equation*}
\begin{split}
\om  & =  \sum_{i=1}^n \frac{1}{k} \cdot (zH_i-G_i) dx_i -\frac{1}{\varphi_{x_1}} \cdot (\varphi H_1+z^{k-1} G_1) dz \\
   & = \sum_{i=1}^n \frac1k \cdot \left[ z\cdot \frac{ \varphi_{x_i} (\varphi H_1 +z^{k-1}G_1) - \varphi_{x_1} z^{k-1}G_i }{\varphi \cdot \varphi_{x_1}} -G_i\right] dx_i -\frac{1}{\varphi_{x_1}}\cdot (\varphi H_1 +z^{k-1}G_1) dz \\
    & = H_1\cdot\frac{1}{k\varphi_{x_1}} \cdot \om_2 -G_1 \cdot  \frac{1}{k\varphi_{x_1}}\cdot \om_1 \\
   & + (z^k+\varphi )\cdot \frac{1}{k\varphi} \cdot \left( \sum_{i=1}^n G_1 \frac{\varphi_{x_i}}{\varphi_{x_1}}dx_i - \sum_{i=1}^n G_i dx_i \right) \\
  & = H_1\cdot \frac{1}{k\varphi_{x_1}} \cdot \om_2 -G_1 \cdot \frac1k \cdot \frac{1}{\varphi_{x_1}} \cdot \om_1 + (z^k+\varphi ) \cdot \frac{1}{k\varphi} \sum_{i=2}^n \frac{G_1 \varphi_{x_i} -G_i\varphi_{x_1}}{\varphi_{x_1}}dx_i.
\end{split}
\end{equation*}

As $\nu_0 (\om )=k-1$, necessarily $\dfrac{G_1}{\varphi_{x_1}}$ must be a unit. So,
$$
k\frac{\varphi_{x_1}}{G_1} \om = \om_1 -\frac{H_1}{G_1} \cdot \om_2 -
(z^k+\varphi) \frac{1}{\varphi G_1} \sum_{i=2}^n \frac{G_1\varphi_{x_i}-G_i \varphi_{x_1}}{\varphi_{x_1}} dx_i,
$$
which is holomorphic by previous considerations. So, we obtain the stated result.
\end{proof}

Let us return to the case $k=2$. Consider a germ of holomorphic foliation in $(\CC^{n+1},\bo )$ generated by an integrable 1-form $\om= \om_1+H\cdot \om_2+(z^2+\varphi )\cdot \om_3$, as in Lemma \ref{cuentasiniciales}. The linear part of $\om$ is $2zdz+\sum_{i=1}^n \dfrac{\partial \varphi_2}{\partial x_i} dx_i$, where $\varphi_2$ is the homogeneous part of degree two of $\varphi$ (which is zero if the order of $\varphi$ is strictly greater than two),
that is not tangent to the radial vector field. Applying Theorem
1 and Corollary 3 of \cite{Loray}, there exists a fibered change
of variables $\Phi_1(\bx,z)= (x_1,x_2,\ldots , x_n,\varphi_1(\bx,z))$, such that
$\Phi_1^{\ast} \FF$ is generated by $\om' = zdz + df_0 (\bx)+z
df_1(\bx ) $, for certain germs of functions $f_0$, $f_1\in \CC \{ \bx\}$. A second change
of variables transforms $\om'$ in $\om''=d \left(
\dfrac{z^2}{2}+f_0(\bx )\right)- f_1(\bx )dz$. Start, now, up to a
constant, with a form $\om''= d(z^2+f_0(\bx ))-f_1(\bx ) dz$. The
integrability condition reads as $df_0\wedge df_1=0$. This means,
by \cite{Mattei-Moussu}, that there exist $f(\bx )$, $h_0(t)$,
$h_1(t)$, such that $f_0(\bx )= h_0(f(\bx ))$,
$f_1(\bx )=h_1(f(\bx ))$. So, if $\rho(\bx,z) =(f(\bx), z)$, then
$\om''=\rho^{\ast} \om_0$, with $\om_0:= d(z^2+h_0(t))-h_1(t)dz$.
Let $r:= \nu_0(h_0) $. The two-dimensional foliation defined by
$\om_0$ has $z^2+t^r+ h.o.t.$ as a separatrix, that is, up to a
unit, equal to $z^2+a(t)z+b(t)=0$. A change of variables
$z\mapsto z-\dfrac{a(t)}{2}$ transforms it in $z^2+c(t)=0$,
$r=\nu_0 (c(t))$. Writing $c(t)=t^r u_0(t)^2$, $u_0(0)\neq 0$,
the separatrix can be written as
$\left(\dfrac{z}{u_0(t)}\right)^2+t^r=0$: a new change of
variables $z\mapsto zu_0(t) $ allows to write this separatrix as
$z^2+t^r=0$. In this case, it is well known (see
\cite{Cerveau-Moussu}) that the foliation is generated by a
1-form $\om_0'=d(z^2+t^r)+ ztA(z,t)\cdot \left(
r\dfrac{dt}{t}-2\dfrac{dz}{z}\right) $.

Collecting all previous considerations, we see that there exists
a change of variables $\Phi_0 (t,z)= (t, zs_1(t)+s_0(t))$, with
$s_1(0)\neq 0$, such that  $\om_0 \wedge \Phi_0^{\ast} \om'_0=0$.
Consider the diagram
$$
\begin{diagram}
\node{\CC^{n+1} } \arrow{e,t}{\rho} \arrow{s,l}{F} \node{\CC^2} \arrow{s,r}{\Phi_0} \\
\node{\CC^{n+1}} \arrow{e,t}{\rho} \node{\CC^2}
\end{diagram}
$$
It is commutative choosing
$F(\bx,z)=(\bx ,zs_1(f(\bx ))+s_0(f(\bx )))$, which is a
diffeomorphism. The form $(F^{-1})^{\ast} \om''$ defines the same
foliation that
$$
\Omega =d(z^2+f(\bx)^r) + z f(\bx )A(z,f(\bx ))\cdot \left(
r\dfrac{df}{f}-2\dfrac{dz}{z}\right),
$$
analytically equivalent to $\FF$, and having $z^2+f(\bx )^r=0$ as
separatrix. Let us observe that the map that transforms $\om$ in
$\Omega$ is of the form $(\bx ,z)\mapsto (\bx,Z(\bx ,z))$, i.e., it
respects the projection over the first $n$ coordinates.
Write $Z(\bx ,z)=\sum_{k=0}^{\infty } Z_k(\bx)z^k$, with
$Z_1(\bo )\neq 0$. There is a unit $U(\bx ,z)$ such that
$Z(\bx ,z)^2+\varphi (\bx )= U(\bx ,z)(z^2+f(\bx )^r)$. If
$U(\bx ,z)=\sum_{k=0}^{\infty } U_k(\bx) z^k$, $U_0(\bo )\neq 0$,
the first coefficients in $z$ of the last equality give the
conditions
\begin{align*}
Z_0(\bx )^2+\varphi (\bx ) & = U_0 (\bx ) f(\bx )^r,\\
2Z_0(\bx ) Z_1(\bx ) & = U_1(\bx ) f(\bx )^r.
\end{align*}

So,
\begin{align*}
\varphi (\bx ) & = U_0(\bx )f(\bx )^r - Z_0(\bx )^2 \\
 & = U_0(\bx ) f(\bx )^r - \frac{U_1(\bx )^2 f(\bx )^{2r}}{4
Z_1(\bx )^2} \\
 & = f(\bx )^r \left[ U_0(\bx )- \frac{U_1(\bx )^2
f(\bx )^r}{Z_1(\bx )^2}\right].
\end{align*}

The expression between brackets turns out to be a unit, as required.

%

\begin{coro}
Every holomorphic codimension one foliation $\FF$ defined by a 1-form $\om$ with linear part equal to $zdz$ has a separatrix. Moreover, every leaf of $\FF$ admits a codimension one foliation with a holomorphic first integral.
\end{coro}

\begin{proof}
From the proof of the Theorem, every codimension one foliation $\FF$ in the conditions of the statement, is the pull-back of a two-dimensional foliation $d(z^2+t^r)+zt A(z,t) \left( r\frac{dt}{t}-2\frac{dz}{z}\right)$, by a holomorphic map $\rho (\bx ,z)=
 (f(\bx) , z)$. This map induces the mentioned foliation in the leaves: the leaves of these foliations are the inverse image of points.
\end{proof}

\begin{nota}

Let us remark that, in previous Corollary, the dicritical case is included (if not, the result would follow from more general results from Cano-Cerveau \cite{Cano-Cerveau}, and Cano-Mattei \cite{Cano-Mattei}). For instance, the foliation defined by
$$
\omega=d(z^2+ \Psi(x,y)^5 )+\Psi (x,y)^2 z \left(5\frac{d\Psi}{\Psi} - 2\frac{dz}{z}\right),
$$
where $\Psi (x,y) = (y^2+x^5)(y^2+2x^5)^2$,
turns out to be dicritical. It is known that if dicritical components are non compact, or if the foliation induced on the compact dicritical components has a meromorphic first integral \cite{RR}, then the foliation has a separatrix. In this example, we are in the first situation. It deserves a closer study to verify if foliations in the conditions of the statement fall into some of these two situations.
\end{nota}

\section{The quasi-homogeneous case}
A polynomial $P(x_1,\ldots ,x_n)= \sum_{I}a_{i_1,\ldots
,i_n}x_1^{i_1}\cdots x_n^{i_n}\in \CC [x_1,\ldots ,x_n]$ is said
to be $(\alpha_1,\ldots , \alpha_n)$-quasi-homogeneous of degree
$d$ if $a_{i_1,\ldots i_n}\neq 0$ implies $\langle
\mathbf{\alpha}, I\rangle= \sum_{k=1}^n \alpha_ki_k =d$.
Equivalently, if the vector field $X=\frac{1}{d} \sum_{k=1}^n
\alpha_k x_k \frac{\partial}{\partial x_k}$ verifies $X(P)=P$.

More generally, a germ of analytic function $f(x_1,\ldots
,x_n)\in \CC \{ x_1,\ldots ,x_n\}$ is called quasi-homogeneous if
there exists a biholomorphism $\Phi : (\CC^n,0)\rightarrow
(\CC^n,0)$ such that $f\circ \Phi$ is a quasi-homogeneous
polynomial. It is well-known (cf. \cite{Saito71}) that the
following conditions are equivalent:
\be
\item $f$ is quasi-homogeneous.
\item There exist a vector field $X$ such that  $X(f)=f$.
\item $f\in Jac (f)$, ideal generated by the partial derivatives
of $f$.
\ee

Let us particularize Theorem \ref{separatriz} to the case when
$z^2+\varphi (\bx )$ is quasi-homogeneous. This means that there
exists a vector field ${\mathcal X}=\sum_{i=1}^n X_i\frac{\partial}{\partial x_i}+ Z \frac{\partial}{\partial z}$
such that  ${\mathcal X}(z^2+\varphi (\bx ))= z^2+\varphi (\bx )$. If, in this
expression, we put $z=0$, we see that $\sum_{i=1}^n X_i (\bx ,0)\varphi_{x_i}=\varphi$, so $\varphi$ is
quasi-homogeneous. In \cite{Cerveau-Mattei}, the following
result is shown:
\begin{lema}
If $\varphi (\bx )$ is quasi-homogeneous and $u(\bx )$ is a unit
with $u(\bo )=1$, the there exists a biholomorphism $\Phi$ such
that $\varphi\circ \Phi =u\varphi$.
\end{lema}

Let us observe that the condition $u(\bo )=1$ is not necessary: if
$\varphi$ is quasi-homogeneous and $c\in \CC\setminus \{ 0\}$,
there exists a biholomorphism $\Phi$ such that  $\varphi\circ
\Phi = c\cdot \varphi$.

So, Theorem \ref{separatriz}, in the quasi-homogeneous case,
implies:
\begin{coro}
Let $\FF$ be a  generalized hypersurface, with quasi-homogeneous
separatrix $z^2+\varphi (\bx )=0$. Then, there exist
coordinates such that  a generator of $\FF$ is
$$
d(z^2+\varphi)+ G(\Psi, z)\cdot (z\cdot \Psi)\cdot \left(
2\frac{dz}{z}-\frac{d\varphi}{\varphi}\right),
$$
where $\varphi = \Psi^p$, $\Psi$ not a power, and $G$ a germ of
holomorphic function in two variables
\end{coro}

\begin{proof}
By Theorem \ref{separatriz}, the foliation is defined by
$$
d(z^2+\varphi')+ G(\Psi, z)\cdot (z\cdot \Psi)\cdot \left(
2\frac{dz}{z}-\frac{d\varphi'}{\varphi'}\right),
$$
where $\varphi$ and $\varphi'$ differ in a unit factor. A
biholomorphism $\Phi$ exists such that $\varphi'\circ
\Phi=\varphi$. Applying this change of variables we obtain
$$
d(z^2+\varphi)+ G(\Psi\circ \Phi, z)+ (z\cdot \Psi\circ
\Phi)\cdot \left( 2\frac{dz}{z}- \frac{d\varphi}{\varphi}\right),
$$
and the result follows.
\end{proof}

\section{A condition for generalized surfaces}
In previous sections we established a pre-normal form for a
generator of a codimension one holomorphic foliation having
$z^2+\varphi (\bx )$ as a separatrix, but this foliation may
eventually be dicritical, or have more separatrices. Let us
briefly return to dimension two. According to \cite{Loray99}, a
1-form having $y^p-x^q$ as a separatrix has the form
$$
\omega =d( y^p-x^q)+\Delta (x,y) (pxdy-qy dx),
$$
where $\Delta (x,y)\in \CC\{ x,y\} $. If $y^p-x^q$ is a
irreducible curve, let $c$ be the conductor of its semigroup.
Then, if $\nu_{(p,q)} (\Delta(x,y) )>c$, the reduction of the
singularities of $\omega$ agrees with the reduction of
$y^p-x^q=0$. Here, $\nu_{(p,q)}\left( \sum_{i,j} a_{ij}
x^i j^j\right) := \min \{ pi+qj;\ a_{ij}\neq 0\}$. In general, if
$\delta = {\rm gcd} (p,q)>1$, define $\nu_{(p,q)}\left(
\sum_{i,j} a_{ij} x^iy^j\right):= \min \left\{ \dfrac{1}{\delta}\cdot
(pi+qj);\ a_{ij}\neq 0\right\}$. Then, the resolution of $\omega$ and
$y^p-x^q$ agrees if\footnote{\cite{Loray99} states that this
condition is \textit{necessary} and sufficient. In fact, it is
not necessary, as examples like $d(y^6-x^3)+axy (6xdy-3ydx)$
show ($a\in \CC^{\ast}$), where a certain arithmetic condition in $a$
is needed, in the style of \cite{Meziani}.}
$\nu_{(p,q)}(\Delta
(x,y))>
\dfrac{(p-1)(q-1)}{\delta}$.

Let us give a version of this result in dimension three, where a
reduction of the singularities is available. In the
quasi-homogeneous case, consider a hypersurface given by the
equation $f(x,y,z)= z^2+\prod_{i=1}^l (y^p-a_ix^q)^{d_i}$, where
$a_i\neq 0$, and $a_i\neq a_j$ if $i\neq j$. Let us observe that
this is note the more general case of quasi-homogeneous surfaces
of our type: such a quasi-homogeneous surface could have an equation $z^2+x^ry^s\prod_{i=1}^l (y^p-a_ix^q)^{d_i}$. We exclude from our considerations the case $r>0$ or $s>0$. Let us briefly describe a reduction of the
singularities of $f(x,y,z)=0$. We will follow the so-called
Weierstrass-Jung method: first of all, the discriminant curve of
the projection $(x,y,z)\mapsto (x,y)$ will  be reduced to normal
crossings, using punctual blow-ups, and then, we will apply usual
procedures in order to desingularize quasi-ordinary surfaces.

Remark that the singular locus of $f(x,y,z)=0$ is
$$
S=\{ \bo\} \cup \{ y^p-a_ix^q=0,\ z=0;\ d_i>1\} .
$$
Denote $r= {\rm gcd} (d_1,d_2,\ldots , d_l)$, $d'_i=d_i/r$,
$d=\sum_{i=1}^l d_i$, $d'_i=\sum_{i=1}^l d'_i$, $\delta
= {\rm gcd} (p,q)$, $\varphi (x,y)= \prod_{i=1}^l
(y^p-a_ix^q)^{d_i}=\Psi (x,y)^r$. The reduction of the
singularities will follow several steps:

\begin{description}
\item[Step I] Blow-up the origin a certain number of times, in
order to reduce to normal crossings the curve $\varphi (x,y)=0$.
Let us see the result in the most interesting chart. For, take
$m$, $n\geq 0$ such that $mp-nq=\delta$, $m$, $n$ minimal (i.e.
$0\leq m<q/\delta$, $0\leq n <p/\delta$). After the
transformation
\begin{align*}
x & \rightarrow x^{p/\delta} y^n \\
y & \rightarrow x^{q/\delta} y^m \\
z & \rightarrow x^{\frac{p+q}{\delta}-1} y^{m+n-1} z,
\end{align*}
the surface is reduced to
$$
x^{2\left( \frac{p+q}{\delta}-1 \right)} y^{2(m+n-1)} \cdot
\left[ z^2+ x^Py^Q \prod_{i=1}^l (y^{\delta} -a_i)^{d_i}\right] ,
$$
where
$$
P= \frac{pq}{\delta}d -2\left( \frac{p+q}{\delta} -1\right),\quad
Q= nqd -2(m+n-1).
$$
The singular locus of the strict transform is given now by the
projective lines given locally by $z=x=0$, $z=y=0$, and the sets of lines $z=0,\
y=a_i^{1/\delta}$.

\smallskip

\item[Step II] Blow-up the lines $z=x=0$, $z=y=0$, a certain
number of times. Let us describe a typical case, namely when $P$,
$Q$ are even numbers. In this case, blow-up $(z=x=0)$ $Q/2$
times, and $(z=y=0)$ $P/2$ times. The equations of this
transform, in a certain appropriate chart, are
\begin{align*}
x & \rightarrow  x\\
y & \rightarrow  y \\
z & \rightarrow  x^{P/2} y^{Q/2} z.
\end{align*}
The strict transform of the surface is now
$$
z^2+\prod_{i=1}^l (y^\delta -a_i)^{d_i}.
$$
\item[Step III] Finally, the lines $z=0$, $y=a_i^{1/\delta}$
with $d_i>1$, must be blown-up, according to the well-known
scheme of the reduction of the singularities of cuspidal
plane curves $z^2+Y^{d_i}=0$.
\end{description}

Let us follow this reduction scheme in the 1-form
$$
\omega =d(z^2+\varphi (x,y))+ G(\Psi, z)\cdot z\Psi\cdot
\left( \frac{d\varphi}{\varphi} - 2 \frac{dz}{z}\right),
$$
with previous notations.

After Step I, the inverse image of $\omega$ is
\begin{equation*}
\begin{split}
x^{2\left( \frac{p+q}{\delta}-1\right)} y^{2(m+n-1)-1}\cdot
\left[ (z^2+x^Py^Q h^r)\omega_1 +xy d(z^2+x^Py^Qh^r) +
\right. \\
\left. +x^a y^b zhG_1\cdot \left( P\frac{dx}{x}+ Q\frac{dy}{y}-2
\frac{dz}{z} + r \frac{dh}{h}\right) \right] ,
\end{split}
\end{equation*}
where:
\begin{align*}
\omega_1 & = 2\left( \frac{p+q}{\delta}-1 \right) ydx +2(m+n-1)
xdy, \\
h(y) & = \prod_{i=1}^l (y^\delta -a_i)^{d'_i}, \\
P & = \frac{pq}{\delta} d -2 \left( \frac{p+q}{\delta}-1\right) , \\
Q & = nqd -2(m+n-1),\\
a & = \frac{pq}{\delta}d' - \left( \frac{p+q}{\delta}- 1\right)
+1,\\
b & = nqd'-(m+n-1)+1.
\end{align*}

Denote $\Omega$ the strict transform of $\omega$, i.e., the
result of dividing by $x^{2\left( \frac{p+q}{\delta} -1\right)
-1} y^{2(m+n-1)-1}$. In Step II, it is necessary to blow-up
several projective lines. Assume, as before, that $P$, $Q$ are even
numbers. In the chart we are working, blow up $P/2$ times the
line $(z=y=0)$, and $Q/2$ times the line $(z=x=0)$. The inverse
image of $\Omega$ is, then,
$$
x^Py^Q\cdot \left[ (z^2+h^r)\cdot \omega_2 + xyd(z^2+h^r)
+ x^{a-\frac{P}{2}} y^{b-\frac{Q}{2}} zhG_2 \left(
-2\frac{dz}{z} + r \frac{dh}{h}\right)\right],
$$
where
$$
\omega_2= \frac{pq}{\delta} d\cdot y dx+nqd\cdot xdy.
$$
The singularities not still reduced are, in these coordinates,
the lines $z=0$, $y^\delta =a_i$, for every $i$ such that
$d_i>1$ (in particular, for all $i$ if $r>0$). Take $\xi$ with
$\xi^\delta =a_i$, and make a translation $y\rightarrow y+\xi$.
Denote $h(y+\xi)= y^{d'_i} H_i (y)$, $H_i (0)\neq 0$. In these
new coordinates, the 1-form turns out to have the expression
\begin{equation*}
\begin{split}
(z^2+y^{d_i} H_i^r) \omega_3 +x(y+\xi) d(z^2 +y^{d_i}H_i^r) +
x^{a-\frac{P}{2}} (y+\xi)^{b-\frac{Q}{2}} zy^{d'_i}H_iG_3 \cdot \\
\left(-2\frac{dz}{z}+ d_i \frac{dy}{y} + r
\frac{dH_i}{H_i}\right) .
\end{split}
\end{equation*}

Assume that $d_i$ is even (the other cases are treated
similarly), and blow up $(z=x=0)$ $d_i/2$ times. The equations
are $z\rightarrow y^{d_i/2} z$, and the final result,
\begin{equation*}
\begin{split}
y^{d_i-1}\cdot \left[ (z^2+H_i^r) \omega_4 +xy (y+\xi)
d(z^2+H_i^r) + x^{a-\frac{P}{2}} (y+\xi)^{b-\frac{Q}{2}}
y^{d'_i-\frac{d_i}{2}+1}\cdot \right. \\
\left. z H_iG_4 \left( -2 \frac{dz}{z} + r
\frac{dH_i}{H_i}\right) \right] ,
\end{split}
\end{equation*}
where
$$
\omega_4 =\frac{pq}{\delta} d (y+\xi) ydx + (nqd\cdot y + d_i
(y+\xi))xdy.
$$
In all these expressions, we denoted $G_1, G_2, G_3, G_4$ the
total transforms of $G(\Psi , z)$ by the succesive
transformations. If $G(\Psi , z)= \sum G_{\alpha \beta}
\Psi^{\alpha} z^{\beta}$, the term $\Psi^{\alpha} z^\beta$ is
transformed in
$$
\left(x^{\frac{pq}{\delta} d'} (y+\xi)^{nqd'} y^{d'_i}
H_i\right)^{\alpha} \cdot \left( x^{\frac12 \frac{pq}{\delta} d}
(y+\xi)^{\frac12 nqd} y^{\frac{d_i}{2}} z
\right)^{\beta}.
$$
In order that the powers that appear are positive, it is
necessary that
\begin{align*}
\frac{pq}{\delta} d' \left(1-\frac{r}{2} \right)  +1 +
\frac{pq}{\delta} d' \left( \alpha + \frac{r}{2} \beta \right) & \geq
0 \\
nqd' \left(1-\frac{r}{2}\right)  +1+nqd' \left( \alpha +
\frac{r}{2} \beta \right) & \geq 0 \\
d'_i \left( 1-\frac{r}{2}\right)   +1 + d'_i \left( \alpha +
\frac{r}{2} \beta \right) & \geq 0.
\end{align*}

These conditions are satisfied if $2\alpha +\beta\geq r-2$. This
condition could be slightly relaxed, including additional
arithmetical conditions as in the two-dimensional case, but we
will not enter into details. As in the two-dimensional case, the effect in the other singular points of the total transform of the foliation is smaller, so, no new restrictions appear.

The cases $P$ even and $Q$ odd, or $P$ odd and $Q$ even are
treated similarly. It is slightly different the case $P$ and $Q$
odd numbers, where it is needed to perform punctual blow-ups at a time of the reduction process (see \cite{FM}), but the final result turns out to be the same. Precise details of the computations are omitted, and will appear in the PhD Thesis of the third author \cite{Hernan}. Collecting previous computations, we
have shown the following result.
\begin{teorema} \label{superficie-generalizada}
Consider a codimension one germ of holomorphic foliation in
$\CC^3$, generated by an integrable 1-form
$$
\omega= d(z^2+\varphi (x,y)) + G(\Psi, z) \cdot (z \Psi ) \left(
\frac{d\varphi}{\varphi} - 2 \frac{dz}{z}\right) ,
$$
where $\varphi (x,y)= \prod_{i=1}^l (y^p-a_i x^q)^{d_i}$, $p$,
$q\geq 2$, $a_i\neq 0$, $a_i\neq a_j$ if $i\neq j$, $r= {\rm gcd}
(d_1,d_2,\ldots , d_l)$, $d'_i= d_i/r$, $\Psi (x,y)=
\prod_{i=1}^l (y^p-a_i x^q)^{d'_i}$, and $G(\Psi, z)=
\sum_{\alpha, \beta} G_{\alpha, \beta} \Psi^{\alpha} z^\beta$ is
a holomorphic function in two variables. Denote $\nu_{(2,r)}(
\psi^{\alpha} z^{\beta}):= (2\alpha+ r\beta)/{\rm gcd} (2,r)$, and $\nu_{(2,r)}
(G)= \min \{ \nu_{(2,r)} (\Psi^{\alpha} z^\beta);\ G_{\alpha
\beta} \neq 0\}$.

Then, if $\nu_{(2,r)} (G)\geq (r-2)/{\rm gcd} (2,r)$, the foliation is a
generalized surface.
\end{teorema}

\section{Acknowledgments}
The authors want to thank the Pontificia Universidad Cat\'{o}lica del Per\'{u} and the Universidad de Valladolid for their hospitality during the visits while preparing this paper.

\end{document}